\documentclass[12pt,a4paper]{article}
\usepackage{}
\usepackage{verbatim}
\usepackage{setspace}
\usepackage[left=2.5cm,right=2.5cm,top=3.0cm,bottom=3.0cm,includeheadfoot]{geometry}
\usepackage{latexsym}
\usepackage{amssymb}
\usepackage{amsthm}
\usepackage{amsmath}
\usepackage{tikz}
\usetikzlibrary{matrix}
\usetikzlibrary{positioning}
\usetikzlibrary{arrows}
\usepackage{graphicx}

\usepackage[english]{babel}

\newtheorem{theorem}{Theorem}[section]
\newtheorem{lemma}[theorem]{Lemma}

\theoremstyle{definition}
\newtheorem{definition}[theorem]{Definition}

\theoremstyle{proposition}
\newtheorem{proposition}[theorem]{Proposition}

\theoremstyle{remark}
\newtheorem{remark}[theorem]{Remark}
\theoremstyle{corollary}
\newtheorem{corollary}[theorem]{Corollary}
\theoremstyle{question}
\newtheorem{question}[theorem]{Question}

\theoremstyle{convention}

\newcommand{\field}[1]{\mathbb{#1}}
\newcommand{\C}{\field{C}}
\newcommand{\R}{\field{R}}

\newcommand{\Z}{\field{Z}}
\newcommand{\Q}{\field{Q}}

\newcommand{\tham}{(M,\omega,T,\phi)}
\newcommand{\sham}{(M,\omega,S^1,\phi)}

\usepackage[latin1]{inputenc}

\title{\textsc{Toric one-skeletons for complexity-one spaces}}
\author{Isabelle Charton}

\begin{document}

	\maketitle

\begin{abstract}
	A complexity-one space is a compact symplectic manifold $(M, \omega)$ endowed with an effective Hamiltonian action of a torus $T$ of dimension $\frac{1}{2}\dim(M)-1$. In this note we prove that for a certain class of complexity-one spaces the Poincar\'e dual of the Chern class $c_{n-1}$ can be represented by a collection of $\frac{n}{2}\chi(M)$ symplectic embedded $2$-spheres, where $\chi(M)$ is the Euler characteristic of $M$ and $\dim(M)=2n$. We call such a collection a toric one-skeleton. The classification of complexity-one spaces is an important subject in symplectic geometry. A nice subcategory of those spaces are the ones which are monotone. The existence of a toric one-skeleton is a useful tool to understand six-dimensional monotone complexity-one spaces. In particular, we will show that the existence of a toric one-skeleton for such a space implies that the second Betti number of $M$ is at most seven. This is a simple application of results by Sabatini-Sepe and Lindsay-Panov.
\end{abstract}

\tableofcontents

	\begin{section}{Introduction}\label{Sec:Intro}
		
	\footnote{2010 Mathematics Subjects Classification. 57R91, 57S25, 37J10 \\
	Keywords and phrase. Complexity-one spaces, Equivariant Cohomology}A symplectic action of a torus \(T=(S^1)^d\) on a symplectic manifold is called \textbf{Hamiltonian}, if it  admits a \textbf{moment map} \(\phi\colon M \rightarrow (\text{Lie}(T))^*\), i.e. \(\phi\) is a smooth map such that
	\begin{align}
	\text{d} \left\langle \phi , \xi \right\rangle = -\iota_{X_\xi} \omega
	\end{align}
	for all \(\xi \in \text{Lie}(T)\), where  \(X_\xi\) is the vector field on \(M\) generated by  \(\xi\).
	We call the quadruple \(\tham\) a \textbf{Hamiltonian \(T\)-space}, if the Hamiltonian action is \textbf{effective} and \(M\) is \textbf{compact}.
	Since the action is assumed to be effective and Hamiltonian, we have \(\dim(T)\leq \frac{1}{2}\dim(M)\). The non-negative integer \(k=\frac{1}{2}\dim(M)-\dim(T)\) is called the \textbf{complexity} of \(\tham\). We also call \(\tham\)  a \textbf{complexity-k space}.  A complexity-zero space is a \textbf{symplectic toric manifold}.
	The classification of Hamiltonian \(T\)-spaces is an important subject in symplectic geometry.
	The Convexity-Theorem - due to Atiyah \cite{atiyah} and Guillemin-Sternberg \cite{GSconvex}  - states that the image of the moment map  \(\phi(M)\) is a convex polytope. Moreover, Delzant \cite{delzant} proved that symplectic toric manifolds are classified by their  moment map image. Also the (equivariant) cohomology ring of a symplectic  toric manifold can be easily  recovered from its  moment map image.
	This is not true  for Hamiltonian \(T\)-spaces with complexity  greater then zero.
	In this paper we study  properties of the second homology groups of complexity-one spaces.
    In particular, a symplectic manifold \((M,\omega)\) always admits  an almost  complex structure \(J\colon TM \rightarrow TM\) which is compatible with the symplectic form \(\omega\), i.e. \(\omega(\cdot,J\cdot)\) is a Riemannian metric on \(M\). Since the space of such structures is contractible, we can define complex invariants of the tangent  bundle \(TM\), for instance Chern classes. \\
    It is known that for a symplectic toric manifold of dimension \(2n\)  the Poincaré dual of the Chern class \(c_{n-1}(M)\in H^{2n-2}(M;\Z)\) can be realised by a collection \(\mathcal{S}\) of \(\frac{n}{2}\chi(M)\) symplectic embedded surfaces, where \(\chi(M)\) is the Euler Characteristic of \(X\).  Namely, the preimage of \(\phi^{-1}(e)\) of each edge \(e\) of the moment map polytope \(\phi(M)\) is a  smoothly embedded, symplectic \(2\)-sphere and \(\mathcal{S}\) is the collection of all these \(2\)-spheres.  We call \(\mathcal{S}\) the \textbf{toric one-skeleton} of the corresponding symplectic toric manifold.
    In \cite{1224} Godinho, Sabatini and von Heymann generalized the concept of toric one-skeletons for Hamiltonian \(T\)-spaces with only isolated fixed points. Similarly, by the next definition we generalize the concept of toric one-skeletons for all  Hamiltonian \(T\)-spaces.
	
	\begin{definition}\label{Def:toric1skelton}
		Let \(\tham\) be a Hamiltonian \(T\)-space of dimension \(2n\) and let \(\chi(M)\) be the Euler Characteristic of \(M\). A \textbf{toric one-skeleton} \(\mathcal{S}\) of \(\tham\) is a collection of  \(\frac{n}{2}\chi(M)\) smoothly embedded, symplectic surfaces, such that the Poincaré dual to their union in \(H_2(M;\Z)\) is the Chern class \(c_{n-1}(M)\), i.e. 
		\begin{align}
		\int_M \mu \cdot c_{n-1}(M)\quad =\quad \sum_{S\in \mathcal{S}} \left( \int_S \mu \right) 
		\end{align}
		for all \(\mu \in H^2(M;\Z)\).
	\end{definition}

	This definition leads to the following question.
	\begin{question}\label{Q:existenceoft1s}
		When does a Hamiltonian \(T\)-space admit a toric one-skeleton?
	\end{question}

We investigate  the following reasons into this question. An important subcategory of Hamiltonian \(T\)-spaces are the ones which are monotone\footnote{A symplectic manifold \((M,\omega)\) is monotone if its first Chern class \(c_1\) is a multiple of the class \([\omega]\) in \(H^2(M;\R)\).}. The classification of monotone Hamiltonian \(T\)-spaces is not solved -even not in dimension six-. By the work of Sabatini-Sepe \cite{SabatiniSepe} and  Lindsay-Panov \cite{LP}  we have that the existence of a toric one-skeleton for a monotone complexity-one  space \(\tham\) of dimension six implies that the second Betti number of \(M\) is at most \(7\) (see  Lemma \ref{Lem:App}). Hence, the existence of a toric one-skeleton is a useful tool to understand six-dimensional monotone complexity-one spaces. 

In this notes we prove that a certain class of complexity-one spaces admit a toric one-skeleton. Now we introduce the class of complexity-one spaces that we will consider.  
Let \(\tham\) be a complexity-one space. Consider the moment map polytope \(\phi(M)\) and let
	 \begin{align*}
	 e= \left( \R\left\langle  \alpha_e\right\rangle  + \beta_e \right) \cap \phi(M) \
	 \end{align*}
	 be an edge of \(\phi(M)\). Then \(F_e=\phi^{-1}(e)\) is a smoothly embedded, symplectic, \(T\)-invariant
	 submanifold with stabilizer \(T_e=\text{exp}(\text{ker}(\alpha_e))\), a codimension-1 subtorus of \(T\). The  \(S^1\cong T/ T_e\)-action on \(F_e\) is effective and  Hamiltonian. Moreover, \(F_e\) has either dimension \(2\) or \(4\). We call \(e\) a \textbf{fat edge}, if the dimension of \(F_e\) is \(4\) (see Section \ref{Sec:MMP}).

	 \begin{definition}\label{Def:extnsion-prop}
	 Let \(\tham\) be a complexity-one space. We say \(\tham\) has the \textbf{Extension-Property \((\mathcal{P}_E)\)} if the following holds :\\
	 Let \(F_e\) be a \(4\)-dimensional isotropic submanifold of \(M\) which is the preimage of  a fat edge \(e\) of the moment map polytope  \(\phi(M)\). Then the \(S^1=T/ T_e\)-action on \(F_e\) extends to an effective Hamiltonian \(T^2\)-action.
	 
	 \end{definition}
	 So the goal of this notes is to prove the following theorem.
	 \begin{theorem}\label{Thm:main}
		A complexity-one space  \(\tham\) that satisfies the Extension-Property \((\mathcal{P}_E)\)  admits a toric one-skeleton.
	\end{theorem}

\begin{remark} \label{Rem:Karshon+Pe}Let \(\tham\) be a complexity-one space.
By Karshon's classification of \(4\)-dimensional Hamiltonian \(S^1\)-spaces
 \cite{12}, we have the following facts:\\
\(\bullet\) If all connected components of \(M^T\) are isolated points, then \(\tham\) satisfies the Extension-Property \((\mathcal{P}_E)\).\\
\(\bullet\) If all connected components of \(M^T\) are  \(2\)-spheres, then \(\tham\) satisfies the Extension-Property \((\mathcal{P}_E)\).\\
On the other side, we have that;\\
\(\bullet\) If \(\tham\) satisfies the Extension-Property \((\mathcal{P}_E)\), then a connected component of \(M^T\) is  an isolated point or a \(2\)-sphere.
\end{remark}
As a consequence of Theorem \ref{Thm:main} we obtain the following corollary.
\begin{corollary}\label{Cor:main}
Let \(\tham\) be a monotone complexity-one space of dimension six, which satisfies the Extension-Property \((\mathcal{P}_E)\). Then the second Betti number of \(M\) is at most 7. 
\end{corollary}

Now we 
explain the idea behind the proof of Theorem \ref{Thm:main}.\\
\textbf{Sketch of the Proof of Theorem \ref{Thm:main}}\\
Our main tool to prove Theorem \ref{Thm:main} is the ABBV Localization formula in equivariant cohomology. In particular, we analyze the ABBV Localization formula for Hamiltonian \(S^1\)-spaces, whose fixed submanifolds have at most dimension \(2\) (see Lemma \ref{Lem:ABBVsurface} and Corollary \ref{Cor:ABBVwith2}).\\
In section \ref{Sec:MMP} we define the pre-toric one-skeleton \(\mathcal{S}_{pre}\) for Hamiltonian \(T\)-spaces. This is a collection of \(T\)-invariant, smoothly embedded, symplectic \(2\)-spheres whose stabilizers are codimensional-\(1\) subtori of \(T\).
Moreover, let \(F_e\) be an isotropic submanifold that belongs to a fat edge of \(\phi(M)\). The Extension-Property \((\mathcal{P}_E)\) implies that \(F_e\) itself admits a 'special' toric one-skeleton  \(\mathcal{S}_{F_e}\). In particular \(\mathcal{S}_{F_e}\) is a collection of \(T\)-invariant \(2\)-spheres. We call
\begin{align*}
\mathcal{\bar{S}}_{F_e}= \left\lbrace  S\in \mathcal{S}_{F_e} \mid S \text{ is not fixed by the \(T\)-action}\right\rbrace 
\end{align*}
a reduced toric one-skeleton of \(F_e\) (see Section \ref{Sec:dim4}).
By  \(M_2^T\) we denote  the set of all \(2\)-dimensional components of \(M^T\). \\
Finally, we will use the ABBV formula, in particular Corollary \ref{Cor:ABBVwith2}, to show that the union 
\begin{align*}
\mathcal{S} = \mathcal{S}_{pre} \bigcup M_2^T \bigcup \left(  \bigcup_{e \text{ is a fat edge} } \mathcal{\bar{S}}_{F_e}  \right) 
\end{align*}
is a toric one-skeleton for \(\tham\).\\

\underline{Acknowledgements.} This work is part of the SFB/TRR 191 `Symplectic Structure in Geometry, Algebra and Dynamics`, funded by the DFG. I would like to thank Silvia Sabatini and Daniele Sepe for very helpful discussions and comments.

\end{section}

\begin{section}{Basic properties of Hamiltonian \(T\)-spaces}\label{Sec:BasicPropHmat}
Let \(T\) be a torus.  We denote its Lie algebra by \(\text{Lie}(T)\) and its lattice by 
\begin{align*}
\ell_T=\ker(\exp\colon \text{Lie}(T)\rightarrow T).
\end{align*}
Moreover, the dual Lie algebra of \(T\) is \((\text{Lie}(T))^*\) and the dual lattice is \(\ell_T^*\).\\
For a Hamiltonian \(T\)-space \(\tham\),  we denote the set of \(T\)-fixed points by \(M^T\) and for any subgroup \(H\) of \(T\) we denote the  set of \(H\)-fixed points by \(M^H\).

\begin{lemma}\label{Lem:Submfd}\cite[Lemma 5.53]{DuffSalm}
Let \(\tham\) be a Hamiltonian \(T\)-space and let \(H\) be a closed subgroup of \(T\). Then any connected component of \(M^H\) is a closed symplectic submanifold of \((M,\omega)\).
\end{lemma}

So we call the connected components of \(M^T\) the \textbf{fixed submanifolds}. Moreover if \(H\) is a subtours of \(T\), then the quotient torus \(T/H\) acts on the connected components of \(M^H\) in Hamiltonian fashion.

Next, we  recall the local normal form near fixed points. Therefore,  we fix a \(T\)-invariant almost complex structure \(J \colon TM \rightarrow TM\) which is compatible with \(\omega\). Hence, for each \(p\in M^T\) we have an \(\C\)-linear \(T\)-representation on \((T_pM,J_p)\cong \C^n\). The weights \(\alpha_{p,1},\dots, \alpha_{p,n}\in \ell_T^*\) of this representation are called the \textbf{isotropic weights} at \(p\). The following theorem by Guillemin-Sternberg \cite{physik} states that small neighborhoods of fixed points are determined by their isotropic weights up to \(T\)-invariant symplectomorphims.

\begin{theorem} \label{locNormForms}(local normal forms near fixed points)
	Let \(\tham\) be a Hamiltonian \(T\)-space of dimension \(2n\). Let \(p\in M^T\) be a fixed point with weights \(\alpha_{p,1},\dots, \alpha_{p,n}\in \ell_T^*\). Then there exists a neighborhood \(U_p\) of \(p\) with complex coordinates \(z_1,\dots,z_n\) centered at \(p\) such that
	\begin{align*}
	&\bullet \text{the symplectic form is }\omega=\frac{i}{2}\sum_{j=1}^{n} \text{d}z_j \wedge \text{d}\bar{z_j},\\
	&\bullet \text{the \(T\)-action is }  \exp(\xi)\cdot(z_1,\dots,z_n)=(\text{e}^{2\pi i \alpha_{p,1}(\xi)}z_1,\dots,\text{e}^{2\pi i \alpha_{p,n}(\xi)}z_n) \text{ for }\xi \in \text{Lie}(T),\\ 
	&\bullet \text{the moment map is } \phi(z)=\frac{1}{2} \sum_{j=1}^{n} \alpha_{p,j} \vert z_j \vert^2  + \phi(p).
	\end{align*}
\end{theorem}
	Note that for each fixed submanifold \(F\) and any \(p,q\in F\) the isotropic weights at \(p\) and \(q\) are equal. So we call these weights the \textbf{isotropic weights at \(F\)} and we denote them by \(\alpha_{F,1},\dots \alpha_{F,n}\). Moreover, the dimension of \(F\) is equal to twice the number of weights \(\alpha_{F,1},\dots \alpha_{F,n}\) which are equal to zero.
	
	\begin{lemma}\label{Lem: ComplexityDim}
	Let \(\tham\) be a complexity-\(k\) space. Then the dimension of each fixed submanifold is at most \(2k\).
	\end{lemma}
\begin{proof}
Let dim(\(M\))=\(2n\), so the complexity of \(\tham\) is \(k=n-d\), where dim(\(T\))=\(d\). Now let \(F\) be a fixed submanifold and let \(\alpha_{F,1}, \dots, \alpha_{F,n}\) be the isotropic weights at \(F\). Since the \(T\)-action on \(M\) is effective, we have that the \(\Z\)-span of \(\alpha_{F,1}, \dots, \alpha_{F,n}\) is equal to \(\ell_T^*\cong \Z^d\). Hence, at most \(k\) of isotropic weights \(\alpha_{F,1}, \dots, \alpha_{F,n}\) at \(F\) are equal to zero.
Since the dimension of \(F\) is equal  to twice the number of isotropic weights at \(F\) that are zero, we conclude \(dim(F)\leq 2k\).
\end{proof}

	\begin{subsection}{The restriction to a subcircle}\label{sec:intermezzo}
	Let \(T\) be a real torus and let \(\bar{\xi} \in \text{Lie}(T)\), such that \(S^1=\exp(\R \bar{\xi})\) is a subcircle of \(T\). Let \(\ell_T^*\) be the dual lattice of \(T\) and \(\ell_{S^1}^*\cong \Z\) be the dual lattice of \(\text{Lie}(S^1) \subset \text{Lie}(T)\). Note that we have a natural restriction map
	\begin{align}
	\Phi \colon \ell_T^* \longrightarrow \ell_{S^1}^* \cong \Z,
	\end{align}
	which is induced by the inclusion \(\text{Lie}(S^1)\rightarrow \text{Lie}(T)\).
	In particular, let \(V\) be a complex vector space of complex dimension \(n\), such that \(T\) acts on \(V\) by \(\text{GL}(V)\) with weights \(\alpha_1,\dots ,\alpha_n \in \ell_T^*\). Then \(S^1\) acts on \(V\) with weights \(\Phi(\alpha_1),\dots, \Phi(\alpha_n )\in \ell_{S^1}^*\).
	\end{subsection}
	
	\begin{definition}\label{Def:generic}
	Let \(\tham\) be a Hamiltonian \(T\)-space. We call \(\bar{\xi}\in \text{Lie}(T)\)  \textbf{generic}, if \(S^1=\exp(\R \bar{\xi})\) is a subcircle of \(T\) and \(M^{S^1}=M^T\). 
 	\end{definition}
 
 \begin{remark}\label{Rem:generic}
 Since we assume a Hamiltonian \(T\)-space to be compact, we have that for any Hamiltonian \(T\)-space there is a generic \(\bar{\xi}\in \text{Lie}(T)\).
 \end{remark}

	\end{section}

\begin{section}{Equivariant Cohomology}\label{Sec:EC}
	
	In this section we review some facts about \(S^1\)-equivariant cohomology. (For a detailed introduction, see for example \cite{AB} and \cite{supersymme}). Moreover, we analyze the ABBV Localization formula for Hamiltonian \(S^1\)-spaces, whose fixed submanifolds have at most dimension \(2\), which is the key ingredient for the proof of Theorem \ref{Thm:main}.

	Let \(M\) be a manifold endowed with an \(S^1\)-action. In the Borel-model the \(S^1\)-equivariant cohomology of  \(M\) is defined as follows. The diagonal action of \(S^1\) on \(M\times S^{\infty}\) is free. By   \(M \times_{S^1} S^{\infty}\) we denote the orbit space. The \(S^1\)-equivariant cohomology ring of \(M\) is
	\begin{align*}
	H_{S^1}^*(M;R)\colon =H^*(M \times_{S^1} S^{\infty};R),
	\end{align*}
	where \(R\) is the coefficient ring.\\
	Let \(M^{S^1}\) be the set of fixed points and assume it is not empty. Let \(F\) be one of its connected components. The inclusion map \(i_F\colon F \rightarrow M\) is an \(S^1\)-equivariant  map, so it induces a map
	\begin{align*}
	i_F^*\colon H_{S^1}^*(M) \rightarrow H_{S^1}^*(F).
	\end{align*}
	Given \(\mu^{S^1}\in H_{S^1}^*(M)\), we write \(\mu^{S^1}\vert_F= i_F^* (\mu^{S^1})\). Moreover if \(P\in M^{S^1}\) is a point, then \(\mu^{S^1}(P)=i_{\left\lbrace P\right\rbrace }^* (\mu^{S^1})\). Note that 
	\begin{align}
	H_{S^1}^*(F)= H^*(F) \otimes H^*(\C P^\infty)=H^*(F)\left[ x\right] ,
	\end{align}
	
	where \(x\) has degree \(2\).
	\begin{remark}
		Let \(P,Q\in F\), where \(F\) is a connected component of \(M^{S^1}\), then \(\mu^{S^1}(P)=\mu^{S^1}(Q)\). In particular, let 
		\begin{align}
		\pi_F\colon H_{S^1}^*(F)= H^*(F) \otimes H^*(\C P^\infty) \longrightarrow  H^*(\C P^\infty)
		\end{align}
		be the natural projection, then \(\pi_F(\mu^{S^1}\vert_F)=\mu^{S^1}(P)\) for all \(P\in F\).
		
	\end{remark}
	Moreover, the projection \(\ M \times_{S^1} S^\infty\ \rightarrow \C P^\infty \) induces a push-forward map in equivariant cohomology
	
	\begin{align}
	H_{S^1}^*(M)\rightarrow H^{*-\operatorname{dim}(M)}(\C P^\infty),
	\end{align}
	which can be seen as integration along the fibers. So we denote it by \(\int_M\). The following theorem, due to Atiyah-Bott and Berline-Vergne (see \cite{AB}, \cite{BV}) gives a formula for the map \(\int_M\) in terms of fixed point set data.
	
	\begin{theorem}\label{Thm:ABBV}(ABBV Localization formula) Let \(M\) be a compact oriented manifold endowed with a smooth \(S^1\)-action. Given \(\mu^{S^1}\in H_{S^1}^* (M;\Q) \)
		\begin{align*}
		\int_M \mu^{S^1} \,=\, \sum_{F\subset M^{S^1}} \int_F \dfrac{\mu^{S^1}\vert_F}{\text{e}^{S^1}(N_F)},
		\end{align*}
		where the sum runs over all connected components \(F\) of \(M^{S^1}\) and \(\text{e}^{S^1}(N_F)\) is the equivariant Euler class of the normal bundle
		\(N_F\) to \(F\).
	\end{theorem}
In particular, if \(F=\left\lbrace P\right\rbrace \) is an isolated fixed point, then
\begin{align}
\int_F \dfrac{\mu^{S^1}\vert_F}{\text{e}^{S^1}(N_F)}= \dfrac{\mu^{S^1}(p)}{(w_{P,1}\cdot \dots \cdot w_{P,n})x^n},
\end{align} 
 where \(w_{P,1}, \dots , w_{P,n}\) are the weights\footnote{Note that the signs of the individual weights are not well-defined, but the sign of their product is. } of the \(S^1\)-representation at \(T_PM\). In the next lemma we analyze the term \(\int_F \frac{(\mu^{S^1}c_{n-1}^{S^1}(M))\vert_F}{\text{e}^{S^1}(N_F)}\) for the case that \(F=\Sigma\)  is a fixed surface.

	\begin{lemma}\label{Lem:ABBVsurface}
		Let \(\sham\) be a Hamiltonian \(S^1\)-space of dimension \(2n\) and let \(\Sigma \) be a fixed surface. Let \(w_1,\dots ,w_{n-1}\) be the weights of \(S^1\)-action on the normal bundle \(N_{\Sigma}\) of \(\Sigma\). Then

		\begin{align*}
		\int_\Sigma \dfrac{(\mu^{S^1}\cdot c_{n-1}^{S^1}(M))\vert_\Sigma}{\text{e}^{S^1}(N_\Sigma)}\quad = \quad \int_\Sigma \mu^{S^1}  \quad  + \quad  2(1-g_\Sigma) \frac{\pi_\Sigma(\mu^{S^1}\vert_\Sigma)}{x}{\sum_{i=1}^{n-1}\frac{1}{w_i}} ,
		\end{align*}
		
		for \( \mu^{S^1} \in H_{S^1}^2(M;\Z) \), where \(g_\Sigma\) is the genus of \(\Sigma\).

	\end{lemma}
	
	\begin{proof}
		We have \(H_{S^1}^*(\Sigma; \Z)=H^*(\Sigma; \Z)\left[ x \right] \). We denote by \(u_\Sigma\) the positive\footnote{with respect to the orientation of \(\Sigma\) induced by the symplectic form \(\omega\)} generator of \(H^2(\Sigma; \Z)\).
		
		First, we compute the  total equivariant Chern class \(c^{S^1}(N_\Sigma)\) of the normal bundle \(N_\Sigma\)
		of \(\Sigma\) in terms of  \(u_\Sigma\) and the weights \(w_1,\dots, w_{n-1}\) . The normal bundle \(N_\Sigma\) is an \(S^1\)-equivariant complex vector bundle over \(\Sigma\) of rank \(n-1\). 
		Due to the fact that each complex vector bundle over a compact surface splits into the direct sum of complex line bundles and that the \(S^1\)-action on \(\Sigma\) is trivial, we have that \(N_\Sigma\) splits into the direct sum of \(S^1\)-equivariant line bundles \(L_1, \dots, L_{n-1}\), such that for all \(j=1,\dots,n-1\)
		the first equivariant Chern class \(c_1^{S^1}(L_j)\) of \(L_j\) is of the form  
		\begin{align}
c_1^{S^1}(L_j)= w_jx+A_j u_\Sigma
		\end{align}
		
		for some \(A_j \in \Z\). Hence, the total \(S^1\)-equivariant Chern class of \(N_\Sigma\) is
		
		\begin{align}
		c^{S^1}(N_\Sigma)=\prod_{i=1}^{n-1}\left( 1+c_1^{S^1}(L_i) \right) =\prod_{i=1}^{n-1}\left( 1+w_ix+A_i u_\Sigma\right).
		\end{align}
		 Note that \(u_\Sigma\cdot u_\Sigma=0\) in \(H_{S^1}^*(\Sigma; \Z)\), so the equivariant Euler class of \(N_\Sigma \) is
		\begin{align}
		e^{S^1}(N_\Sigma)=c_{n-1}^{S^1}(N_\Sigma)=\prod_{i=1}^{n-1}\left( w_ix+A_i u_\Sigma\right) =\prod_{i=1}^{n-1}\left( w_ix\right) \left( 1+\sum_{i=1}^{n-1} \dfrac{A_i}{w_ix}u_\Sigma\right) 
		\end{align}
		
		and the formal inverse of \(e^{S^1}(N_\Sigma)\) is
		\begin{align}
		\dfrac{1}{e^{S^1}(N_\Sigma)}= \dfrac{1}{\prod_{i=1}^{n-1}\left( w_ix\right)} \left( 1-\sum_{i=1}^{n-1} \dfrac{A_i}{w_ix}u_\Sigma\right).
		\end{align}
		Next, we have
		\begin{align}
		c_{n-2}^{S^1}(N_\Sigma)=\sum_{j=1}^{n-1}\left( \prod_{i=1, i\neq j}^{n-1}\left( w_ix+A_i u_\Sigma\right)\right)  =\sum_{j=1}^{n-1} \left( \prod_{i=1, i\neq j}^{n-1}\left( w_ix\right) \left( 1+\sum_{i=1,i\neq j}^{n-1} \dfrac{A_i}{w_ix}u_\Sigma\right) \right). 
		\end{align}
		Moreover, since the \(S^1\)-action on \(T\Sigma\) is trivial, the total equivariant Chern class of \(T\Sigma\) is \(c^{S^1}(\Sigma)=1+2(1-g_\Sigma)u_\Sigma\). Now we use \(c^{S^1}(M)\vert_\Sigma=c^{S^1}(\Sigma) c^{S^1}(N_\Sigma)\), so we obtain
		\begin{align}
		c_{n-1}^{S^1}(M)\vert_\Sigma = e^{S^1}(N_\Sigma) + 2(1-g_\Sigma)   \left( \sum_{j=1}^{n-1} \prod_{i=1, i\neq j}^{n-1}\left( w_i\right)\right) x^{n-2}u_\Sigma
		\end{align}
		
		and
		
		\begin{align}
		\dfrac{c_{n-1}^{S^1}(M)\vert_\Sigma}{e^{S^1}(N_\Sigma)} =1 + \quad 2(1-g_\Sigma) \dfrac{u_\Sigma}{x} {\sum_{i=1}^{n-1} \frac{1}{w_i}}.
		\end{align}
		
		Now let \(\mu^{S^1}\in H_{S^1}^*(M,\Z)\), then 
		
		\begin{align}
		\mu^{S^1}\vert_\Sigma= \mu\vert_\Sigma + \pi_\Sigma(\mu^{S^1}\vert_\Sigma),
		\end{align}
		 where \(\mu\in H^2(M;\Z)\) is image of \(\mu^{S^1}\) under the restriction map \(H_{S^1}^2(M;\Z)\rightarrow H^2(M;\Z) \). We conclude that 
		\begin{align}
		\dfrac{\mu^{S^1}\cdot c_{n-1}^{S^1}(M)\vert_\Sigma}{e^{S^1}(N_\Sigma)} \quad = \quad \mu \vert_\Sigma +\quad 2(1-g_\Sigma) \dfrac{\pi_\Sigma(\mu^{S^1}\vert_\Sigma)}{x}\cdot u_\Sigma \left( \sum_{i=1}^{n-1}\dfrac{1}{w_i}\right) \quad + \quad \pi_\Sigma(\mu^{S^1}\vert_\Sigma).
		\end{align}
		
		So the claim follows.
		
	\end{proof}
	
	As a consequence of the ABBV Localization formula and of Lemma \ref{Lem:ABBVsurface}, we obtain the following corollary.
	
	\begin{corollary}\label{Cor:ABBVwith2}
		Let \(\sham\) be a Hamiltonian \(S^1\)-space of dimension \(2n\), such that \(M^{S^1}\) consists just of fixed \(2\)-spheres \(\Sigma_1,\dots,\Sigma_N\) and isolated fixed points \(P_1,\dots, P_K\). For all \(i=1,\dots,N\), let \(w_{\Sigma_i,1},\dots, w_{\Sigma_i,n-1} \) be the weights of the \(S^1\)-action on the normal bundle of  \(\Sigma_i\). Moreover, for all \(j=1,\dots,K\), let \(w_{P_j,1},\dots, w_{P_j,n} \) be the weights of the \(S^1\)-action on \(T_{P_j}M\). Then

		\begin{align}
		\int_M \mu^{S^1}\cdot c_{n-1}^{S^1}(M)  = & \sum_{i=1}^{N}\left( \int_{\Sigma_i} \mu^{S^1} + 2 \dfrac{\pi_{\Sigma_i}(\mu^{S^1}\vert_{\Sigma_i})}{x} \left( \frac{1}{w_{\Sigma_i,1}}+\dots+ \frac{1}{w_{\Sigma_i,n-1}}\right) \right) \\
		& + \sum_{j=1}^{K}  \left(  \dfrac{\mu^{S^1}(P_j)}{x} \left( \frac{1}{w_{P_j,1}}+\dots+ \frac{1}{w_{P_j,n}}\right)\right)  
		\end{align}
		
		for \(\mu^{S^1}\in H_{S^1}^2(M;\Z)\)
	\end{corollary}

\begin{proof}
	For an isolated fixed point \(P_j\), we have \(e^{S^1}(N_{P_j})=  
	\left( \prod_{s=1}^{n} w_{P_j,s}\right) x^n \) and
	
	\begin{align*}
	\mu^{S^1}\cdot c_{n-1}^{S^1}(M) \vert_{P_j}= \mu^{S^1}(P_j)\left(  
	\sum_{h=1}^{n} \left( \prod_{k=1, k\neq h}^n w_{P_j,k}\right) \right)  
	x^{n-1}.
	\end{align*}
	Hence,
	
	\begin{align*}
	\frac{\mu^{S^1}\cdot c_{n-1}^{S^1}(M) \vert_{P_j}}{e^{S^1}(N_{P_j})} =  
	\dfrac{\mu^{S^1}(P_j)}{x} \left( \frac{1}{w_{P_j,1}}+\dots+  
	\frac{1}{w_{P_j,n}}\right).
	\end{align*}
	So the claim follows from the ABBV formula and Lemma \ref{Lem:ABBVsurface}.
	
\end{proof}

\end{section}

\begin{section}{The moment map polytope}\label{Sec:MMP}
	The aim of this section is to introduce conventions for the fixed point data of a Hamiltonian \(T\)-space and to define the pre-toric one-skeleton.
	Therefore, we recall facts about the moment map polytope (i.e. the image of the moment map) for a Hamiltonian \(T\)-space. 
	One of most the important theorems for Hamiltonian \(T\)-spaces is the famous Convexity-Theorem proved by Atiyah \cite{atiyah} and Guillemin-Sternberg \cite{GSconvex}.

\begin{theorem}(Convexity-Theorem)\label{Thm:Convexity}
Let \(\tham\) be a Hamiltonian \(T\)-space, then the fibers of \(\phi\) are connected and the image of the moment map \(\phi(M)\) is a convex polytope. In particular, \(\phi(M)\) is the convex hull of the images of the fixed submanifolds.
\end{theorem}
Another important property of the moment map is stated in the following theorem.

\begin{theorem}
	(Sjamaar \cite{Sjamaar})\label{Thm:MMisopen}
Let \(\tham\) be a Hamiltonian \(T\)-space. As a map to \(\phi(M)\) the moment map is open.
\end{theorem}

The moment map polytope contains informations about the corresponding Hamiltonian \(T\)-space. In particular, by the Convexity-Theorem we have that the preimage of a vertex of \(\phi(M)\) is a connected component of \(M^T\).
Moreover, a face of the moment  map polytope gives us a symplectic submanifold. Namely, let \(\mathcal{F}\) be a \(d\)-dimensional face of the polytope \(\phi(M)\), then \(\mathcal{F}\) is of the form 
\begin{align*}
\mathcal{F}=\left( \mathcal{H}_\mathcal{F} +\left\lbrace  p_\mathcal{F}\right\rbrace \right)  \bigcap \phi(M),
\end{align*}
where \(p_\mathcal{F}\) is a point in \(\mathcal{F}\) and \(\mathcal{H}_\mathcal{F}\) is a linear subspace of \((\text{Lie}(T))^*\) of dimension \(d\). We call \(\mathcal{H}_\mathcal{F}\) the defining subspace for \(\mathcal{F}\).
\begin{proposition}\label{Prop:faceofMMP}
Let \(\tham\) be a Hamiltonian \(T\)-space and let \(\mathcal{F}\) be a \(d\)-dimensional face of \(\phi(M)\) with defining  subspace \(\mathcal{H}_\mathcal{F}\). Then

\begin{align}
T_\mathcal{F}= \exp \left( \ker \left( \mathcal{H}_\mathcal{F}\right) \right) 
\end{align} 
is a subtorus of \(T\) of codimension \(d\) and the preimage \(\phi^{-1}(\mathcal{F})\) is a connected component of \(M^{T_\mathcal{F}}\). Moreover, \(\phi^{-1}(\mathcal{F})\) is a symplectic submanifold of \((M,\omega)\)  and the quotient torus \(T/T_\mathcal{F}\) acts effectively and in Hamiltonian fashion   on \(\phi^{-1}(\mathcal{F})\).

\end{proposition}

A prove of Proposition \ref{Prop:faceofMMP} can be found in \cite[Chapter 3]{Nico}.
The next corollary gives an upper bound for the dimension of \(\phi^{-1}(\mathcal{F})\).

\begin{corollary}\label{Cor:DimFace}
Let \(\tham\) be  a complexity-\(k\) space and let \(\mathcal{F}\) be a \(d\)-dimensional face of \(\phi(M)\) . Then \(\phi^{-1}(\mathcal{F})\) is a symplectic submanifold of dimension at most \(2(k+d)\).
\end{corollary}

\begin{proof}
Let \(\mathcal{H}_\mathcal{F}\) be the defining subspace of \(\mathcal{F}\). By Proposition \ref{Prop:faceofMMP}, we have that \(T_\mathcal{F}= \exp \left( \ker \left( \mathcal{H}_\mathcal{F}\right) \right)\) is a subtorus of \(T\) of codimension \(d\). Moreover, \(T_\mathcal{F}\) acts on \((M,\omega)\)  in Hamiltonian  fashion with complexity

\begin{align}
k_\mathcal{F}&=\dfrac{1}{2} \dim (M)-\dim(T_\mathcal{F})\\
&=\dfrac{1}{2} \dim (M)-\left( \dim(T)-d\right) \\&=k+d.
\end{align}
Moreover, by Proposition \ref{Prop:faceofMMP} we also have that \(\phi^{-1}(\mathcal{F})\) is a connected component of \(M^{T_\mathcal{F}}\). 
And so Lemma \ref{Lem: ComplexityDim} implies that 
\begin{align}
\dim (\phi^{-1}(\mathcal{F})) \leq 2k_\mathcal{F}=2(k+d).
\end{align}

\end{proof}

\begin{subsection}{Conventions for the fixed point data }

Let \(\tham\) be a Hamiltonian \(T\)-space of dimension \(2n\) and \(p \in M^T\) an isolated fixed point with isotropic  weights \(\alpha_{p,1} ,\dots, \alpha_{p,n} \in \ell_T^* \setminus\left\lbrace 0\right\rbrace  \). For \(i=1,\dots, n \quad \)  \(T_{\alpha_{p,i}}=\exp(\ker (a_{p,i}))\) is a codimension-\(1\) subtorus of \(T\). We denote by \(S_{\alpha_{p,i}}\) the connected component of \(M^{T_{\alpha_{p,i}}}\) that contains \(p\). So \(S_{\alpha_{p,i}}\) is a symplectic and \(T\)-invariant submanifold.

\begin{definition}\label{Def:lightweight}
We call \(S_{\alpha_{p,i}}\)  the submanifold that belongs to the weight \(\alpha_{p,i}\). We say the weight \(\alpha_{p,i}\) is \textbf{light} if the dimension of \(S_{\alpha_{p,i}}\) is \(2\) (or equivalent \(\alpha_{p,i}\) and \(\alpha_{p,j}\) are linearly independent for all \(j\in \left\lbrace 1,\dots, \hat{i},\dots ,n\right\rbrace \) ). Otherwise the weight \(\alpha_{p,i}\) is called \textbf{heavy}.
\end{definition}
Moreover, we classify the isolated fixed points of a Hamiltonian \(T\)-space into two categories as follows.

\begin{definition}\label{Def:GKMpoint}
Let \(\tham\) be a Hamiltonian \(T\)-space. We denote the set of isolated fixed points by \(M_0^T\).  An isolated fixed point \(p\in M_0^T\) is called \textbf{GKM} if the isotropic weights at \(p\) are pairwise linearly independent (i.e. all isotropic weights are light), otherwise \(p\) is called \text{non-GKM}.  We denote by \textbf{\(M_{GKM}^T\)} the set of all GKM-points.
\end{definition}

\begin{remark}\label{Rem:forGKMpoints}
A Hamiltonian \(T\)-space is called GKM if for each codimension-\(1\) subtorus \(H\) of \(T\) any connected component of \(M^H\) has at most dimension \(2\) -or equivalent the fixed point set \(M^T\) contains just isolated fixed points, which are all GKM-. Hamiltonian GKM-spaces are a subcategory of GKM spaces, which were introduced by Goresky-Kottwitz-MacPherson in \cite{GKM}.
\end{remark}

\begin{lemma}\label{Lem:locnonGKM}
Let \(\tham\) be a complexity-one space of dimension \(2n\) and let \(p\in M^T\) be an isolated fixed  point which is not GKM. Then exactly two weights of \(p\) are heavy.  Moreover, we have
\begin{align*}
& \phi(p) \text{ is a vertex of } \phi(M) \text{ and } \phi^{-1}(\phi(p))=p\quad \quad \quad \text{or} \\
&\phi(p) \text { is not a vertex of }\phi(M), \text{ but } \phi(p) \text{ lies on an edge of } \phi(M).
\end{align*}
\end{lemma}
\begin{proof}
Since the \(T\)-action is effective, we have that \(\Z-\text{span}\left\lbrace \alpha_{p,1},\dots,\alpha_{p,n}\right\rbrace =l^*(\cong \Z^{n-1})\).
So w.l.o.g we assume that the weights are ordered by  the following convention: 
\begin{align}
&\bullet \R - \text{span}\left\lbrace \alpha_{p,1} ,\dots, \alpha_{p,n-1}\right\rbrace = (\text{Lie}(T))^*\\
&\bullet \alpha_{p,n} = \lambda \cdot \alpha_{p,n-1} \text{ for some } \lambda \in \Q\setminus \left\lbrace 0\right\rbrace 
\end{align}
Hence, the weights \(\alpha_{p,1},\dots,\alpha_{p,n-2}\) are light  and the weights \(\alpha_{p,n-1},\alpha_{p,n}\) are heavy. Moreover, Theorems 4.1, 4.2 and 4.3 imply that
\begin{align*}
&\bullet \text{If } \lambda>0, \text{ then } \phi(p) \text{ is a vertex of } \phi(M) \text{ and } \phi^{-1}(\phi(p))=p.\\
&\bullet \text{If } \lambda<0, \text{ then } \phi(p) \text { is not a vertex of }\phi(M), \text{ but } \phi(p) \text{ lies on an edge of } \phi(M).
\end{align*}
\end{proof}

\begin{definition}\label{Def:fatedge}
Let \(\tham\) be a complexity-one space. An edge \(e\) of the moment map polytope is called \textbf{fat} if \(\phi^{-1}(e)\) has dimension \(4\).   We denote by \(E_{\Phi(M)}^{fat}\) the set of fat edges of the moment map polytope \(\Phi(M)\).
\end{definition}

\begin{lemma}\label{Lem:meetingfatedge}
Let \(\tham\) be a complexity-one space of dimension \(2n\).\\
\underline{\textbf{\((i)\)}} If \(p\in M^T\) is a GKM isolated fixed point, then \(\phi(p)\) does not meet a fat edge of \(\phi(M)\).\\
\underline{\textbf{\((ii)\)}} If \(p\in M^T\) is a non-GKM isolated fixed point, then \(\phi(p)\) meets precisely one fat edge of \(\phi(M)\). In particular, let \(a_{p,n-1}, \alpha_{p,n}\) be the heavy weights of \(p\), then this fat edge is given by
\begin{align*}
\left( \phi(p) + \R\left\langle \alpha_{p,n-1}\right\rangle \right) \cap \phi(M) = \left( \phi(p) + \R\left\langle \alpha_{p,n}\right\rangle \right) \cap \phi(M) 
\end{align*}
\underline{\textbf{\((iii)\)}} If \(\Sigma\subset M^T\) is a fixed surface. Then all \(n-1\) edges of \(\phi(M)\) that are meeting at the vertex \(\phi(\Sigma)\) are fat. Moreover, these edges are given by 
\begin{align*}
\left( \phi(\Sigma) + \R\left\langle \alpha_{\Sigma,j}\right\rangle \right) \cap \phi(M) 
\end{align*}
for \(i=1,\dots, n-1\), where \(\alpha_{\Sigma,1},\dots, \alpha_{\Sigma,n-1} \) are the weights of the normal bundle of \(\Sigma\).
\end{lemma}

\begin{subsection}{The pre-toric one-skeleton}
	
	Let \(\tham\) be a Hamiltonian \(T\)-space of dimension \(2n\). Let \(p\in M^T\) be an isolated fixed point and let \(\alpha \) be a light weight at \(p\). The submanifold \(S_\alpha\)  that belongs to \(\alpha\) is a symplectic and \(T\)-invariant surface. Moreover, \(S_\alpha\) is fixed by the codimension-\(1\) subtours \(T_\alpha=\exp(\ker \alpha)\) and the action of the quotient circle \(T/T_\alpha\) on \(S_\alpha\) is effective and Hamiltonian. Since each Hamiltonian \(S^1\)-space of dimension \(2\) is a \(2\)-sphere with exactly \(2\) fixed points, we have that \(S_\alpha\) is a \(2\)-sphere and it contains precisely two points  of \(M^T\), namely \(p\) and the other  one is denoted by \(q\). Moreover, since \(T\) acts on \(TpS_\alpha\) with weight \(\alpha\), we have that \(T\) acts on \(TqS_\alpha\) with weight -\(\alpha\). In particular, -\(\alpha\) is a light weight at \(q\). 
	
	\begin{definition}\label{Def:pretoric1skelton}
	Let \(\tham\) be a Hamiltonian \(T\)-space. The pre-toric one-skeleton \(\mathcal{S}_{pre}\) of \(\tham\) is 
	the set of all submanifolds which belong to a light weight of an isolated fixed point.
	\end{definition}
	
	\begin{remark}\label{Rem:pre}
	From the discussion above, we have that \(\mathcal{S}_{pre}\) is a collection of smoothly embedded and symplectic \(2\)-spheres. Moreover, let \(p\) be an isolated fixed point and let \(k\) be the number of light weights at \(p\). Then there are precisely \(k\) spheres of \(\mathcal{S}_{pre}\) meeting at \(p\).
	\end{remark}

	Together with the discussion above, we conclude that the per-toric one-skeleton of a complexity-\(1\) space satisfies the following properties:
	
	\begin{lemma}\label{Lemma:proofpre}
	Let \(\tham\) be a complexity-one space of dimension \(2n\). Let \(M_0^T\) be the set of isolated fixed points and \(M_{GKM}^T\) be the set of GKM isolated fixed points.\\
	\(\bullet\) If \(p\in M_{GKM}^T\), then there are precisely \(n\) spheres of \(\mathcal{S}_{pre}\) meeting at \(p\), namely the submanifolds that belong to the  weights of \(p\). \\
	\(\bullet\) If \(p\in M_0^T\setminus M_{GKM}^T\), then there are precisely \((n-2)\) spheres of \(\mathcal{S}_{pre}\) meeting at \(p\), namely the submanifolds that belong to the light weights of \(p\).\\
	Moreover, the cardinality of \(\mathcal{S}_{pre}\) is 
	
	\begin{align}
	\vert \mathcal{S}_{pre} \vert \quad = \quad \frac{1}{2} n \vert M_{GKM}^T \vert  + \frac{1}{2} (n-2) \vert M_0^T\setminus M_{GKM}^T \vert.
	\end{align}
	\end{lemma}

\begin{proposition}\label{Pro:pre}
Let \(\tham\) be a complexity-one space of dimension \(2n\). For an isolated fixed point \(p\in M_0^T\) let 
\(\alpha_{p,1} ,\dots, \alpha_{p,n}\) be the weights of \(p\), where the weight ordered such that \(\alpha_{p,1} ,\dots, \alpha_{p,n-2}\)  are the light weights of \(p\) if \(p\) is not GKM. Let \(\bar{\xi}\in \text{Lie}(T)\) generic and \(S^1=\exp(\R \bar{\xi})\), then

\begin{align}
\sum_{S \in \mathcal{S}_{pre}} \int_S \mu^{S^1} =& \sum_{p\in M^T_{GKM} }\dfrac{\mu^{S^1}(p)}{x} \left( \frac{1}{\Phi(\alpha_{p,1})}    +\dots +\frac{1}{\Phi(\alpha_{p,n})}      \right) \quad +\\
& \sum_{p\in M^T_0 \setminus M^T_{GKM} }\dfrac{\mu^{S^1}(p)}{x} \left( \frac{1}{\Phi(\alpha_{p,1})}    +\dots +\frac{1}{\Phi(\alpha_{p,n-2})}      \right)
\end{align}

for all \(\mu^{S^1}\in H_2^*(M;\Z)\), where \(\Phi\colon \ell_T^*\rightarrow \ell_{S^1}^*\) is the restriction.

\end{proposition}

\begin{proof}
Let \(S\in \mathcal{S}_{pre}\) with \(S\cap M^T=\left\lbrace p,q\right\rbrace \), then 
\(T\) acts on \(TpS\) with weight \(\alpha \in \ell^* \setminus \left\lbrace 0\right\rbrace \), where  \(\alpha\) is a light
weight at \(p\) and  \(T\) acts on \(TqS\) with weight -\(\alpha\), which is also a light weight at \(q\).
So this implies that \(S^1\cong \exp(\R \bar{\xi})\) acts on \(T_pS\) with weight \(\Phi(\alpha)\)  and on \(T_qS\) with weight \(-\Phi(\alpha)=\Phi(-\alpha)\).\\
Now let \(\mu^{S^1}\in H_{S^1}^2(M;\Z)\), then ABBV formula gives us 
\begin{align}
\int_S \mu^{S^1} =\dfrac{\mu^{S^1}(p)}{\Phi(\alpha)\cdot x} \quad + \quad \dfrac{\mu^{S^1}(q)}{\Phi(-\alpha)\cdot x}. 
\end{align}
The claim follows form this and Lemma \ref{Lemma:proofpre}.
\end{proof}

\end{subsection}

\end{subsection}

\end{section}

\begin{section}{Hamiltonian \(T\)-spaces of dimension \(4\)}\label{Sec:dim4}
Let \((F^4, \omega,T^2, \psi)\) be a symplectic toric manifold of dimension \(4\). Then \(\psi(F)\) is a Delzant  polytope and \((F, \omega,T, \psi)\) is determined by this polytope (up to \(T^2\)-equivariant symplectomorphims) (see \cite{delzant}). We fix a splitting \(T^2=S^1\times S^1\). So we identify \(\text{Lie}(T)=\R^2\) and \((\text{Lie}(T))^*=(\R^2)^*=\R^2\). Let \(E_{\psi(F)}\) be the set of edges of \(\psi(F)\). Then

\begin{align}
S_e =\psi^{-1}(e)
\end{align}

is a symplectic embedded \(2\)-sphere for all \(e \in E_{\psi(F)}\). Moreover, \(\left\lbrace S_e \mid e \in E_{\psi(F)}\right\rbrace  \) is a toric  one-skeleton for \((F^4, \omega,T^2, \psi)\).\\
If we restrict the action to the subcircle \(S^1=\left\lbrace 1\right\rbrace \times S^1 \subset T^2\), we get an \(4\)-dimensional Hamiltonian \(S^1\)-space \((F,\omega,S^1, \phi)\), where \(\phi= \text{pr}_2 \circ \psi\)
and \( \text{pr}_2\colon \R^2 \rightarrow \R \) is the projection onto the second factor. Moreover, we have\\
\(\bullet\) If \(e\in E_{\psi(F)}\) is a horizontal edge, then \(S_e\) is fixed by the \(S^1\)-action.\\
\(\bullet\) If \(e\in E_{\psi(F)}\) is not a horizontal edge, then let \(v_1\) and \(v_2\) be the vertices of \(\psi(F)\) connected by this edge. W.l.o.g. we assume that \(\text{pr}_2(v_1)<\text{pr}_2(v_2)\). Let \(\frac{k_e}{b_e}\) the slope of this edge\footnote{
	\(k_e ,b_e \in \Z_{>0}  \) with \(gcd(k_e,b_e)=1\)  if \(e\) is not a vertical edge. Otherwise the slope is \(\infty\) and \(k_e=1\).

} and \(p_i=\phi^{-1}(v_i)\) for \(i=1,2\). Then  \(S^1\) acts on \(S_e\) with fixed points \(p_1\) and \(p_2\) and \(k_e\) is the  weight at \(p_1\) and \(-k_e\) is the weight at \(p_2\). Moreover, let \(\mu^{S^1}\in H_{S^1}^2(F;\Z)\) then the ABBV formula gives us 
\begin{align}\label{Eq:1}
\int_{S_e} \mu^{S^1}= \dfrac{\mu^{S^1}(p_1)-\mu^{S^1}(p_2)}{x\cdot k_e}.
\end{align}

\begin{definition}\label{Def:Reducedtoric1sekelton}
	Let \((F,\omega, S^1, \phi)\) be a Hamiltonian \(S^1\)-space  of dimension \(4\) such that the \(S^1\)-action extends to an effective Hamiltonian \(T^2\)-action. Let \(\psi\colon F \rightarrow (\text{Lie}(T^2))^*\) a moment map for the \(T^2\)-action. We call
\begin{align*}
\mathcal{\bar{S}}_F = \left\lbrace S= \psi^{-1}(e) \mid e \in E_{\psi(F)}  \right\rbrace  \setminus F^{S^1}_{2}
\end{align*}
the \textbf{reduced toric one-skeleton } of \((F, \omega, S^1, \phi)\) (with respect to this extension of  the \(S^1\)-action), where \(F_2^{S^1}\) is the set of 2-dimensional components of \(F^{S^1}\).

\end{definition}

\begin{lemma}\label{Lem:IntRedtoric1seleton}
	Let \((F,\omega, S^1, \phi)\) be a Hamiltonian \(S^1\)-space  of dimension \(4\) such that the \(S^1\)-action extends to an effective Hamiltonian \(T^2\)-action and let \(\mathcal{\bar{S}}_F\) be the reduced toric one-skeleton. Moreover, let \(F_0^{S^1}\) be the set of isolated fixed point and \(F_2^{S^1}\) be the set of fixed \(2\)-spheres. Then

\begin{align*}
	\sum_{S_\in \mathcal{\bar{S}}_F} \int_{S} \mu^{S^1} = 2 \sum_{\Sigma \in F_2^{S^1}}  \mathcal{E}_\Sigma \dfrac{\pi_\Sigma(\mu^{S^1}\vert_\Sigma)}{x} \quad + \quad\sum_{p \in F_0^{S^1}}
		\dfrac{\mu^{S^1}(p)}{x}\left( \frac{1}{\lambda_{p,1}}+\frac{1}{\lambda_{p,2}}\right),
		  \end{align*}	

for \(\mu^{S^1}\in H_{S^1}^2(M;\Z)\), where \(\lambda_{p1}\) and \(\lambda_{p2}\) are the weights of the \(S^1\)-action at \(p\in F_0^{S^1}\) and
\begin{align*}
\mathcal{E}_\Sigma \colon = \begin{cases}
+1 \text{ if is \(\phi\) attains its minimum at } \Sigma,\\
-1 \text{  if is \(\phi\) attains its maximum at } \Sigma.
\end{cases}
\end{align*}

\end{lemma}
\begin{proof}
Let \(p\) be an isolated fixed point. Then there are precisely two \(2\)-spheres \(S_{p1}, S_{p,2}\) of \(\mathcal{\bar{S}}_F\)  meeting at \(p\), each belongs to one weight at \(\lambda_{p,1}, \lambda_{p,2}\). In particular, \(S^1\) acts on \(T_{p_i}S_{p,1}\) with weight \(\lambda_{p,i}\) for \(i=1,2\).\\
 For a fixed surface \(\Sigma\) there are precisely two \(2\)-spheres \(S_{\Sigma,1}, S_{\Sigma,2}\) of \(\mathcal{\bar{S}}_F\) that intersect with \(\Sigma\). Moreover, \(\Sigma\cap S_{\Sigma,i}\) contains one point \(p_{\Sigma,i}\) and \(S^1\) acts on \(T_{p_{\Sigma,i}}S_{\Sigma,i}\) with weight \(+1\) resp. \(-1\) if \(\phi\) attains its minimum reps. maximum at \(\Sigma\) .\\
 Now for \(S\in \mathcal{\bar{S}}_F\) and  \(p\in S\cap F^{S^1}\) let \(\lambda_{S,p}\) be the weight of \(S^1\)-action \(T_pS\), then the ABBV formula gives us
 
 \begin{align*}
 \int_{S} \mu^{S^1} = \sum_{p\in S \cap F^{S^1}} \left( \dfrac{\mu^{S^1}(p)}{x} \dfrac{1}{\lambda_{S,p}}\right) .
 \end{align*}
 We conclude 
 \begin{align*}
 \sum_{S_\in \mathcal{\bar{S}}_F} \int_{S} \mu^{S^1} &= \sum_{S_\in \mathcal{\bar{S}}_F} \left(  \sum_{p\in S \cap F^{S^1}} \left( \dfrac{\mu^{S^1}(p)}{x} \dfrac{1}{\lambda_{S,p}}\right) \right) \\
&= 2 \sum_{\Sigma \in F_2^{S^1}} \mathcal{E}_\Sigma \dfrac{\pi_\Sigma(\mu^{S^1}\vert_\Sigma)}{x} \quad + \quad\sum_{p \in F_0^{S^1}}
 \dfrac{\mu^{S^1}(p)}{x}\left( \frac{1}{\lambda_{p,1}}+\frac{1}{\lambda_{p,2}}\right).
 \end{align*}

\end{proof}

\begin{subsection}{Reduced toric one-skeletons of fat edges}
	Let \(\tham\) be a complexity-one space of dimension \(2n\) with \(\bar{\xi} \in \text{Lie}(T)\) generic.
	Let 
	\begin{align}
	e= \left\lbrace \R\left\langle \alpha_e\right\rangle + \beta_e \right\rbrace \cap \phi(M)
	\end{align}
	 be a fat edge of the moment map polytope and \(F_e=\phi^{-1}(e)\) the corresponding isotropic submanifold of dimension \(4\). We have two circle actions on \(F_e\), namely the one of the subcircle \(\exp(\R \bar{\xi})\) and the one of the quotient circle \(T/\exp(\ker(\alpha_e))\). In order to distinguish them, we denote by \(S^1\) the subcircle \(\exp(\R \bar{\xi})\) and by \(\mathcal{C}\) the quotient circle \(T/\exp(\ker(\alpha_e))\).

	 We assume that \(\alpha_e \in \ell_T^*\) is primitive. If \(p\in M^T\cap F_e\), then \(T\) acts on \(T_pF_e\) with weights \(\lambda_{p,1}\alpha_e\) and \(\lambda_{p,2}\alpha_e\), where \(\lambda_{p,1},\lambda_{p,2} \in \Z\). Hence, the quotient \(\mathcal{C}\) acts on \(T_pF_e\) with weights \(\lambda_{p,1},\lambda_{p,2} \in \Z\). Moreover let \(\Phi\colon \ell_T^*\rightarrow \ell_{S^1}^*\) the restriction map, then \(S^1\) acts on \(T_pF_e\) with weights \(\Phi(\lambda_{p,1}\alpha_e)\) and \(\Phi(\lambda_{p,2}\alpha_e)\).

	\begin{lemma}\label{Lem:IntRedtoric1seleton2}
	Assume that the \(\mathcal{C}\)-action on \(F_e\) extends to an effective Hamiltonian \(T^2\)-action on \(F_e\). Let \(\mathcal{\bar{S}}_{F_e}\) the corresponding reduced toric one-skeleton. For \(p\in M^T\cap F_e\) let \(\alpha_{p,n-1},\alpha_{p,n}\) be the heavy weights of the \(T-\)action on \(T_pM\) and for \(\Sigma \in M_2^T \cap F_e\) let \(\alpha_{\Sigma,e}\) be the weight  of the \(T-\)action on the normal bundle of \(\Sigma\) in \(F_e\). Then
	
	\begin{align*}
	\sum_{S_\in \mathcal{\bar{S}}_{F_e}}\int_{S} \mu^{S^1} = 2 \sum_{\Sigma \in M_2^{S^1}\cap F_e}\dfrac{\pi_\Sigma(\mu^{S^1}\vert_\Sigma)}{x} \dfrac{1}{\Phi (\alpha_{\Sigma e})} \quad + \quad\sum_{p \in M_0^{S^1}\cap F_e}
	\dfrac{\mu^{S^1}(p)}{x}\left( \frac{1}{\Phi(\alpha_{p,n-1})}+\frac{1}{\Phi(\alpha_{p,n})}\right),
	\end{align*}	
	
	for \(\mu^{S^1}\in H_{S^1}^2(M;\Z)\).
	\end{lemma}
The proof of Lemma \ref{Lem:IntRedtoric1seleton} is the same as those of Lemma \ref{Lem:IntRedtoric1seleton2}.
Next, we consider what happen if we pick for each fat edge a reduced toric  one-skeleton.

	
	\begin{proposition}\label{Pro:Int3}
		
		Let \(\tham\) be a complexity-one space of dimension \(2n\) which satisfies the Extension-Property \((\mathcal{P}_E)\). For each fat edge \(e\in E_{\phi(M)}^{fat}\), we pick one reduced toric one-skeleton \(\mathcal{\bar{S}}_{F_e}\)  of \(F_e\).\\
		For a fixed \(2\)-sphere \(\Sigma \in M_2^T\) let \(\alpha_{\Sigma,1}, \dots , \alpha_{\Sigma,n-1}\) be the weights of the normal bundle of \(\Sigma\) and for a non-GKM isolated  fixed point  \(p \in M_0^T\setminus M_{GKM}^T\) let \(\alpha_{p,n-1} , \alpha_{p,n}\) be the heavy weights of \(p\). Then

		\begin{align}
	\sum_{e\in E_{\phi(M)}^{fat}} \left( \sum_{S \in \hat{\mathcal{S}}_{F_e}} \int_S \mu^{S^1} \right) =& 2\sum_{\Sigma \in M_2^T }\dfrac{\pi_\Sigma(\mu^{S^1}\vert_\Sigma)}{x} \left( \frac{1}{\Phi(\alpha_{\Sigma,1})}    +\dots +\frac{1}{\Phi(\alpha_{\Sigma,n-1})}      \right) \quad +\\
		& \sum_{p\in M^T_0 \setminus M^T_{GKM} }\dfrac{\mu^{S^1}(p)}{x} \left( \frac{1}{\Phi(\alpha_{p,n-1})}   +\frac{1}{\Phi(\alpha_{p,n})}      \right)
		\end{align}
		
		for \(\mu^{S^1}\in H_{S^1}^2(M;\Z)\),  where \(\bar{\xi}\in \text{Lie}(T)\) is generic, \(S^1=\exp(\R \bar{\xi})\) and \(\Phi\colon \ell_T^*\rightarrow \ell_{S^1}^*\) is the restriction.
		
	\end{proposition}

\begin{proof}
For a fixed \(2\)-sphere \(\Sigma\) there are precisely \((n-1)\) fat edges meeting at the vertex \(\phi(\Sigma)\), each belongs to a weight \(\alpha_{\Sigma,1}, \dots , \alpha_{\Sigma,n-1}\) of the normal bundle of \(\Sigma\) in \(M\). In particular,
\begin{align}
\left\lbrace  \alpha_{\Sigma,1}, \dots , \alpha_{\Sigma,n-1}\right\rbrace = 
\left\lbrace    \alpha_{\Sigma,e}   \mid e \in E_{\phi(M)}^{fat}  \text{ such that  \(\phi(\Sigma)\)   is a vertex of \(e\) }                                                         \right\rbrace ,
\end{align}
where \(\alpha_{\Sigma,e}\) is the weight of the normal bundle of \(\Sigma\) in \(F_e\).\\
For a non-GKM point \(p\in M_0^T\setminus M_{GKM}^T\), we have \(p\in F_e\) for precisely one \(e\in E_{\phi(M)}^{fat}\). In particular, 

\begin{align*}
M_0^T\setminus M_{GKM}^T  = \bigcup_{e \in E_{\phi(M)}^{fat}} \left( M_0^T \cap F_e\right) 
\end{align*}
Together with Lemma \ref{Lem:IntRedtoric1seleton2} the claim follows.
\end{proof}

\begin{lemma}\label{Lem:Cardi}
Let \(\tham\) be a complexity-one space of dimension \(2n\) that satisfies the Extension-Property \((\mathcal{P}_E)\). For a fat edge  \(e\in E_{\phi(M)}^{fat}\), we pick one reduced toric one-skeleton \(\mathcal{\bar{S}}_{F_e}\)  of \(F_e\). Then the cardinality of \(\cup_{e \in E_{\phi(M)}^{fat}}  \mathcal{\bar{S}}_{F_e}\) is 
\begin{align}
\vert \cup_{e \in E_{\phi(M)}^{fat}}  \mathcal{\bar{S}}_{F_e} \vert =  \vert M_0^T \setminus     M_{GKM}^T \vert + (n-1) \vert M_2^T \vert. 
\end{align}

\end{lemma}

\end{subsection}

\end{section}

\begin{section}{Proof of Theorem \ref{Thm:main} and Corollary \ref{Cor:main}}\label{Proofs}

In this section we prove the Theorem \ref{Thm:main}. Moreover, in Lemma \ref{Lem:App} we prove that the existence  of a toric one-skeleton  for a monotone complexity-one space of dimension six implies that the second Betti number is bounded by \(7\). Finally, we see that Corollary \ref{Cor:main} is a simple consequence of Theorem \ref{Thm:main} and Lemma \ref{Lem:App}.
So let us prove Theorem \ref{Thm:main}.

\begin{proof}[Proof of Theorem  \ref{Thm:main} ]

Let \(\tham\) be a complexity-one space of dimension \(2n\) that satisfies the Extension-Property \((\mathcal{P}_E)\).
Hence, the connected components of \(M^T\) are isolated points or \(2\)-spheres. \\
\(\bullet\) For each fat edge \(e\in E^{fat}_{\phi(M)}\), we pick one reduced toric one-skeleton \(\mathcal{\bar{S}}_{F_e}\) of \(F_e=\phi^{-1}(e)\).\\
\(\bullet\) Let \(\mathcal{S}_{pre}\) be the pre-toric one-skeleton of \(\tham\).\\
\(\bullet\) Let \(M_2^T\) be the set of fixed  \(2\)-spheres of \(M^T\).\\

Now we want to show that 

\begin{align}
\mathcal{S}&=\mathcal{S}_{pre} \bigcup M_2^T \bigcup \left(     \bigcup_{e\in E^{fat}_{\phi(M)}}   \bar{\mathcal{S}}_{F_e}          \right)
\end{align}

is a toric one-skeleton of \(\tham\).
Therefore, let \(\bar{\xi}\in Lie(T)\) generic, \(S^1\cong \exp(\R \bar{\xi})\) and let \(\Phi \colon \ell_T^* \rightarrow \ell_{S^1}^*\).

Then Proposition \ref{Pro:pre} and \ref{Pro:Int3} imply

\begin{align*}
\sum_{S \in \mathcal{S}}   \int_S \mu^{S^1} = &\\                      
&\sum_{\Sigma \in M_2^T }\left(  \int_{\Sigma} \mu^{S^2} +2\dfrac{\pi_\Sigma(\mu^{S^1}\vert_\Sigma)}{x} \left( \frac{1}{\Phi(\alpha_{\Sigma,1})}    +\dots +\frac{1}{\Phi(\alpha_{\Sigma,n-1})}      \right)\right) \\
&+ \sum_{p\in M^T_0 }\dfrac{\mu^{S^1}(p)}{x} \left( \frac{1}{\Phi(\alpha_{p,1})}    +\dots +\frac{1}{\Phi(\alpha_{p,n})}      \right)
 \end{align*}

for all \(\mu^{S^1}\in H_{S^1}^*(M;\Z)\). By Corollary \ref{Cor:ABBVwith2} , this is equal to \(\int_M \mu^{S^1}c_{n-1}^{S^1}\).

  Moreover, since the fixed point set \(M^T\) is torsion-free the restriction map 
\(H_{S^1}^*(M;\Z)\rightarrow H^*(M;\Z)\) is surjective (\cite[Kirwan]{Kirwan}). Let \(\mu \in H^2(M;\Z)\), then \(\mu\) admits an \(S^1\)-equivariant extension \(\mu^{S^1}\in H_{S^1}^2(M;\Z)\). We obtain
\begin{align}
\sum_{S\in \mathcal{S}} \int_S \mu=\sum_{s\in \mathcal{S}} \int_S \mu^{S^1} = \int_M \mu^{S^1} c_{n-1}^{S^1} =\int_M \mu c_{n-1}.
\end{align}
Hence, the union of the \(2\)-spheres of \(\mathcal{S}\) in \(H_2(M;\Z)\) is the Poincaré dual to the Chern class \(c_{n-1}\).\\
Moreover,  by Lemma \ref{Lemma:proofpre} and \ref{Lem:Cardi} we have

\begin{align}
\vert \mathcal{S} \vert = \dfrac{1}{2}n \cdot   \left( \vert M_0^T \vert + 2 \vert M_2^T \vert \right). 
\end{align}
So it is left to show that the Euler Characteristic \(\chi(M)\) is equal to \(\vert M_0^T \vert + 2 \vert M_2^T \vert\).
Therefore let \(\phi^{\bar{\xi}}\colon M \rightarrow \R \) be the \(\bar{\xi}\)-component of  the moment map. Since \(\bar{\xi}\) is generic, we have that \(\phi^{\bar{\xi}}\) is a perfect Morse-Bott function whose critical submanifolds are precisely the connected components of \(M^{S^1}=M^T\) (Kirwan \cite{Kirwan}). Hence,

\begin{align*}
H^*(M;\R)= \bigoplus_{F \subset M^T} H^{*-2d_F}(F;\R),
\end{align*}
where \(d_F\) is the number of negative weights of the normal bundle of \(F\). Since all connected components of \(M^T\) are isolated points or \(2\)-spheres, we have 

\begin{align*}
\chi(M)= \vert M_0^T \vert + 2 \vert M_2^T \vert.
\end{align*}

\end{proof}

\begin{lemma}\label{Lem:App}
Let \(\tham\) be a complexity-one space of dimension six, which is monotone and admits a toric one-skeleton \(\mathcal{S}\). Then the second Betti number of \(M\) is at most \(7\).
\end{lemma}

\begin{proof}
By a result of Sabatini-Sepe \cite{SabatiniSepe}, the odd Betti numbers of \(M\) are equal to zero. So let \(b_0,\dots,b_6\) be the Betti numbers of \(M\). Since \(M\) is connected we have \(b_0=1\). Moreover, by the Poincaré duality we have \(b_0=b_6\) and \(b_2=b_4\). Hence, the Euler Characteristic of \(M\) is 
\begin{align}
\chi(M)=b_0-b_1+b_2-b_3 +b_4-b_5+b_6=2(1+b_2).
\end{align}
Since each element \(S \in \mathcal{S}\) is a smoothly embedded, symplectic surface of \((M,\omega)\) we have 
\begin{align}
\int_S \omega >0  \quad \text{for all }S \in \mathcal{S}.
\end{align}
Now let \(c_1\in H^2(M;\Z)\) be the first Chern class of \((M,\omega)\). Since \((M,\omega)\) is monotone, the integer
\(\int_S c_1\) is positive for all \(S\in \mathcal{S}\). This implies
\begin{align*}
\sum_{S\in \mathcal{S}} \int_{S} c_1 \geq \vert \mathcal{S} \vert = \frac{3 \cdot \chi(M)}{2}= 3(1+b_2).
\end{align*}

Since \(\mathcal{S}\) is the Poincaré dual of the second Chern class \(c_2 \in H^4(M;\Z)\), we have 

\begin{align}
\int_M c_1c _2 = \sum_{S \in  \mathcal{S} }c_1 \geq 3(1+b_2)
\end{align}
Moreover, by the work of Lindsay-Panov \cite{LP} we have that \(\int_M c_1c_2 =24\).
Hence,
\begin{align}
24= \int _M c_1c_2 \geq 3(1+b_2),
\end{align}
and this implies \(b_2 \leq 7\).

\end{proof}

\end{section}

Mathematisches Institut, Universitat zu Koeln, Weyertal 86-90, D-50931, 
Koeln, Germany\\
E-mail address: icharton@math.uni-koeln.de	
	
\end{document}